\documentclass[10pt, leqno]{amsart}
\usepackage{amsfonts,amssymb}
\allowdisplaybreaks[4]

\usepackage{amsmath}
\usepackage{amsthm}
\usepackage{amsrefs}
\usepackage{qsymbols}
\usepackage{latexsym}
\usepackage{chngcntr}
\usepackage[noadjust]{cite}
\usepackage{paralist}
\usepackage{url}
\usepackage{fixltx2e}
\usepackage{esint}

\newtheorem{theorem}{Theorem}[section]

\newtheorem{lemma}[theorem]{Lemma}
\newtheorem{proposition}[theorem]{Proposition}

\theoremstyle{definition}
\newtheorem{definition}[theorem]{Definition}
\newtheorem{remark}[theorem]{Remark}

\newtheorem{innercustomgeneric}{\customgenericname}
\providecommand{\customgenericname}{}
\newcommand{\newcustomtheorem}[2]{%
  \newenvironment{#1}[1]
  {%
   \renewcommand\customgenericname{#2}%
   \renewcommand\theinnercustomgeneric{##1}%
   \innercustomgeneric
  }
  {\endinnercustomgeneric}
}

\newcustomtheorem{assumption}{Assumption}

\newcommand{\IR}{\mathbb{R}}

\newcommand{\IN}{\mathbb{N}}
\newcommand{\IZ}{\mathbb{Z}}

\newcommand{\IB}{\mathbb{B}}
\newcommand{\IT}{\mathbb{T}}

\newcommand{\cF}{\mathcal{F}}
\newcommand{\cS}{\mathcal{S}}

\newcommand{\cT}{\mathcal{T}}

\newcommand{\C}{\mathrm{C}}

\newcommand{\bla}{{\mathbf{a}}}
\newcommand{\blb}{{\mathbf{b}}}

\newcommand{\e}{\mathrm{e}}

\renewcommand{\d}{\mathrm{d}}

\newcommand{\divergence}{\operatorname{div}}

\DeclareMathOperator{\Esup}{ess\ sup}

\numberwithin{equation}{section} 

\title[Strong time-periodic solutions to the bidomain equations]{Strong time-periodic solutions to the bidomain equations with arbitrary large forces}
\author{Yoshikazu Giga}
\author{Naoto Kajiwara}
\author{Klaus Kress}
\address{Graduate School of Mathematical Sciences, The University of Tokyo, 3-8-1 Komaba, Meguro, Tokyo, 153-8914, Japan}
\address{Graduate School of Mathematical Sciences, The University of Tokyo, 3-8-1 Komaba, Meguro, Tokyo, 153-8914, Japan}
\address{Fachbereich Mathematik, Technische Universit\"at Darmstadt, Schlossgartenstr. 7, 64289 Darmstadt, Germany}
\email{labgiga@ms.u-tokyo.ac.jp}
\email{kajiwara@ms.u-tokyo.ac.jp}
\email{kkress@mathematik.tu-darmstadt.de}
\thanks{This work was partly supported by the DFG International Training Group IRTG 1529 and the JSPS Japanese-German Graduate Externship on Mathematical Fluid Dynamics. 
The first author is partly supported by JSPS through grant Kiban S (No.\,26220702), Kiban A (No.\,17H01091), Kiban B (No.\,16H03948). 
The second author is a research assistant supported by Kiban S (No. 26220702). The third author is supported by the DFG International Research Training Group 1529 on Mathematical Fluid Dynamics at TU Darmstadt }


\begin{document}

\begin{abstract}
We prove the existence of strong time-periodic solutions to the bidomain equations with arbitrary large forces. 
We construct weak time-periodic solutions by a Galerkin method combined with Brouwer's fixed point theorem and a priori estimate independent of approximation. 
We then show their regularity so that our solution is a strong time-periodic solution in $L^2$ spaces. 
Our strategy is based on the weak-strong uniqueness method. 
\end{abstract}
\maketitle

\section{Introduction}\label{Sec: Introduction}

In this paper, we consider the bidomain system which is a well-established system describing the electrical wave propagation in the heart. 
The system is given by
\begin{align}\tag{BDE}\label{Eq: BDE}
\left\{
\begin{aligned}
\partial_t u - \divergence (\sigma_i \nabla u_i) + f(u,w) &= s_i & \mathrm{in} & ~ (0 , \infty) \times \Omega ,  \\
\partial_t u + \divergence (\sigma_e \nabla u_e) + f(u,w) &= - s_e & \mathrm{in} & ~ (0 , \infty) \times \Omega ,  \\
\partial_t w + g(u,w) &= 0 & \mathrm{in} & ~ (0 , \infty) \times \Omega ,  \\
u &= u_i - u_e & \mathrm{in} & ~ (0 , \infty) \times \Omega ,  \\
\sigma_i \nabla u_i \cdot \nu = 0 , ~ &\sigma_e \nabla u_e \cdot \nu = 0 & \mathrm{on} & ~ (0 , \infty) \times \partial \Omega ,  \\
u(0) = u_0 , ~ &w(0) = w_0 & \mathrm{in} & ~ \Omega. 
\end{aligned}
\right.
\end{align}
Here $\Omega \subset \IR^d$ denotes a domain describing the myocardium and the outward unit normal vector to $\partial \Omega$ is denoted by $\nu$. 
The unknown functions $u_i$ and $u_e$ model the intra- and extracellular electric potentials, and $u$ denotes the transmembrane potential. 
The variable $w$, the so-called gating variable, corresponds to the ionic transport through the cell membrane. 
The anisotropic properties of the intra- and extracellular tissue parts are described by the conductivity matrices $\sigma_i(x)$ and $\sigma_e(x)$, whereas $s_i(t,x)$ and $s_e(t,x)$ denote the intra- and extracellular stimulation current, respectively. The ionic transport is described by the nonlinearities $f$ and $g$. In this article, we will consider a large class of ionic models including those by FitzHugh--Nagumo, Rogers--McCulloch, and Aliev--Panfilov. Note, that we will look at the Aliev--Panfilov model in a slightly modified form as considered, e.g., in~\cite{GO}. The {\em FitzHugh--Nagumo model} reads as
\begin{align*}
f(u,w) &= u (u-a) (u-1) + w = u^3 - (a+1) u^2 +au + w,\\
g(u,w) &= -\varepsilon (ku-w),
\end{align*}
with $0<a<1$ and $k$, $\varepsilon >0$.\\
In the {\em Rogers--McCulloch model} the functions $f$ and $g$ are given by
\begin{align*}
f(u,w)&= bu(u-a)(u-1) + uw = bu^3 - b(a+1)u^2 + bau +uw,\\
g(u,w)&= - \varepsilon (ku -w),
\end{align*}
with $0<a<1$ and $b$, $k$, $\varepsilon >0$.\\
For the modified {\em Aliev--Panfilov model} we have
\begin{align*}
f(u,w) &= bu(u-a)(u-1) + uw = bu^3 - b(a+1)u^2 + bau +uw,\\
g(u,w) &= \varepsilon (ku(u-1-d) +w)
\end{align*}
with $0<a,d<1$ and $b$, $k$, $\varepsilon >0$. To get weak time-periodic solutions for the Aliev--Panfilov model we will need the further assumption $b > k$.
For a detailed description of the bidomain model we refer to the monographs by Keener and Sneyd~\cite{MR1673204} and Colli Franzone, Pavarino, and Scacchi~\cite{MR3308707}.

Since the bidomain model describes electrical activities in the heart, it is a natural question to ask whether it admits time-periodic solutions. 
Therefore, consider the situation where the bidomain model is innervated by periodic intra- and extracellular stimulation currents $s_i$ and $s_e$. 
Recently, Hieber, Kajiwara, Kress, and Tolksdorf~\cite{HKKT} proved the existence and uniqueness of a strong $T$-periodic solution to the innervated model in real interpolation spaces provided the external forces $s_i$ and $s_e$ are both time-periodic of period $T>0$. 
In their approach, they furthermore assumed that the external forces satisfy a suitable {\em smallness condition}. 

It is the goal of this paper to prove the existence of time-periodic solutions without assuming any smallness condition on the external forces. 
We employ the method given by Galdi, Hieber, and Kashiwabara~\cite{GHK} for the case of the primitive equations. 
First, the existence of weak time-periodic solutions is shown by using a Galerkin approximation combined with Brouwer's fixed point theorem. 
Then, we use the global well-posedness result by Colli Franzone and Savar\'e~\cite{MR1944157} and consider the {\it weak} time-periodic solution as a {\it weak} solution to the initial value problem. 
Finally, we apply a weak-strong uniqueness argument to get a strong-time periodic solution without assuming any smallness condition for the external applied currents.

The bidomain model was first introduced by Tung~\cite{Tun78} in 1978. 
Despite its central importance in cardiac electrophysiology, the rigorous mathematical analysis started not until the work of Colli Franzone and Savar\'e~\cite{MR1944157} in 2002. 
They introduced a variational formulation of the bidomain problem and  proved the existence and uniqueness of weak and strong solutions to the bidomain equations with FitzHugh--Nagumo type nonlinearities. 
A slightly more detailed review of their results is given in Section~\ref{Sec: Strong periodic solution}. 
Veneroni~\cite{MR2474265} extended their results to more general ionic models including the Luo and Rudy I model. 
In 2009, Bourgault, Cordi\`ere, and Pierre~\cite{MR2451724} presented a new approach to the bidomain system. 
They introduced the so-called bidomain operator within the $L^2$-setting and showed that it is a non-negative and selfadjoint operator. 
By using the bidomain operator, they transformed the bidomain system into an abstract evolution equation and showed the existence and uniqueness of a local strong solution and the existence of a global weak solution for a large class of ionic models including the three models introduced above.
Later, Kunisch and Wagner~\cite{KW13} showed uniqueness and further regularity for these weak solutions. 
Giga and Kajiwara~\cite{GK} investigated the bidomain system within the $L^p$-setting and showed that the bidomain operator is the generator of an analytic semigroup on $L^p(\Omega)$ for $p\in (1,\infty]$. 
Recently, Hieber and Pr\"uss proved the maximal $L_p$-$L_q$ regularity for the bidomain operator in \cite{HP2} and proved the global well-posedness in \cite{HP}. 
This paper considered the case $s_{i,e}=0$ and only FitzHugh-Nagumo type non-linearities. 
The general case was treated in \cite{Kajiwara}. 
However, to use the global well-posedness results in \cite{Kajiwara}, we need higher regularity for $s_{i,e}$ compared with \cite{MR1944157}. 
More recently, the bidomain equations were treated as a kind of  gradient system in \cite{BC}. 
They proved the global well-posedness results in $L^2$ spaces and energy spaces. 
Their paper also treated the case $s_{i,e}=0$. 

For results concerning the dynamics of the solution, we refer to the work of Mori and Matano~\cite{MM16}. They studied the stability of front solutions of the bidomain equations. 

On a microscopic level, the cardiac cellular structure is described by two disjoint domains $\Omega_i$ and $\Omega_e$, which denote the intra- and extracellular space, respectively, and which are separated by the the active membrane $\overline{\Gamma}= \partial \Omega_i  \cap \partial \Omega_e$. 
The intra- and extracellular quantities are defined on the corresponding domains and the transmembrane potential $u$ is a function on $\overline{\Gamma}$. 
After a homogenization procedure, see, e.g.,~\cite{MR1944157, PSC05}, the macroscopic model of the bidomain equations is obtained. 
Here all membrane, intra-, and extracellular quantities are defined everywhere on $\Omega$.

This paper is organized as follows: We start in Section~\ref{Sec: Preliminaries} with collecting known facts concerning the bidomain operator.  
In Section~\ref{Sec: Weak periodic solution}, we construct a weak time-periodic solution to the bidomain equations. 
When we construct the weak time-periodic solutions, we need some growth conditions on the nonlinear terms $f,g$. 
Fortunately, all three models mentioned above fulfill these conditions, which are confirmed in the appendix. 
A global well-posedness result is reviewed in Section~\ref{Sec: Strong periodic solution} and invoked to obtain a strong time-periodic solution for the FitzHugh--Nagumo model. 
We do not treat the other two models introduced above since the global well-posedness for the {\it initial value problem} is not proved in a suitable $L^2$ setting. 

\section{Preliminaries}\label{Sec: Preliminaries}

In this section, we fix some notation and formally introduce the bidomain operator in a weak as well as in a strong setting. In the whole article, let $\Omega \subset \IR^d$ denote a bounded domain whose boundary $\partial \Omega$ is of class $C^2$. For convenience, we use the following notation for the function spaces which we will use throughout this article
\begin{align*}
 V=H^1(\Omega),\  H=L^2(\Omega), \ V'=(H^1(\Omega))',
\end{align*}
where the spaces are endowed with their usual norms and they are Hilbert spaces. 
Furthermore, we set $Q=(0,T)\times\Omega$. 
The canonical pair of $V'$ and $V$ is denoted by ${}_{V'}\langle \cdot, \cdot \rangle_V$. 

We assume that the conductivity matrices $\sigma_i$ and $\sigma_e$ satisfy the following assumption.
\begin{assumption}{C}
\label{Ass: Conductivity matrices}
The conductivity matrices $\sigma_i , \sigma_e : \overline{\Omega} \to \IR^{d \times d}$ are symmetric matrices and are functions of class $C^1 (\overline{\Omega})$. 
Ellipticity is imposed by means of the following condition: there exist constants 
$\underline{\sigma}$, $\overline{\sigma}$ with $0 < \underline{\sigma} < \overline{\sigma}$ such that 
\begin{align}
\underline{\sigma} |\xi|^2 \leq {}^t\xi\sigma_i (x)\xi \leq \overline{\sigma}|\xi|^2 \quad \text{and} \quad \underline{\sigma} |\xi|^2 \leq {}^t\xi\sigma_e (x)\xi \leq \overline{\sigma}|\xi|^2 \label{Eq: UE}
\end{align}
for all $x \in \overline{\Omega}$ and all $\xi \in \IR^d$. 
Moreover, it is assumed that 
\begin{align} \begin{aligned}
\sigma_i \nabla u_i \cdot \nu = 0 \qquad &\Leftrightarrow \qquad \nabla u_i \cdot \nu =0 \quad &{\rm on}~\partial\Omega, \\
 \sigma_e \nabla u_e \cdot \nu = 0 \qquad &\Leftrightarrow \qquad \nabla u_e \cdot \nu =0 \quad &{\rm on}~\partial\Omega. \label{Eq: EV}
\end{aligned}
\end{align}
\end{assumption}
It is known due to~\cite{franzone1990mathematical} that~\eqref{Eq: EV} is biological reasonable.

First, we want to introduce the bidomain operator in a weak setting as well as the corresponding bidomain bilinear form. Therefore, we define $V_{av} (\Omega):= \{u\in V : \int_\Omega u \; \d x = 0\}$. Following~\cite{MR2451724}, we define the bilinear forms
\begin{align*}
a_i(u,v) := \int_{\Omega} \sigma_i \nabla u \cdot\nabla v \; \d x, \quad a_e(u,v) := \int_{\Omega} \sigma_e \nabla u \cdot\nabla v \; \d x
\end{align*}
for all $(u,v)\in V_{av}\times V_{av}$. Due to~\eqref{Eq: UE} these bilinear forms are symmetric, continuous and uniformly elliptic on $V_{av}\times V_{av}$. Then, we define the weak operators $A_i$ and $A_e$ from $V_{av}$ onto $V_{av}'$ by
\begin{align*}
\langle A_i u ,v \rangle := a_i (u,v), \quad \langle A_e u ,v \rangle := a_e (u,v)
\end{align*}
for all $(u,v)\in V_{av}\times V_{av}$. Let $P_{av}$ be the orthogonal projection from $V$ to $V_{av}(\Omega)$, i.e., $P_{av}u := u - \frac{1}{|\Omega|} \int_{\Omega} u \: \d x$ and denote its transpose by $P_{av}^T : V_{av}'\rightarrow V'$. Now we are able to define the weak bidomain operator and the corresponding bidomain bilinear form as
\begin{align*}
A&= P_{av}^T A_i (A_i + A_e)^{-1} A_e P_{av}, \\
a(u,v) &= \langle Au , v \rangle
\end{align*} 
for all $(u,v)\in V\times V$. We have the following lemma.

\begin{lemma}[{\cite[Theorem 6]{MR2451724}}]\label{Lemma: coercive}
The bidomain bilinear form $a(\cdot, \cdot)$ is symmetric, continuous and coercive on $V$,
\begin{align*}
\alpha \|u\|_V^2 &\leq a(u,u) + \alpha \|u\|_H^2, \quad &{\mbox for~all~} u\in V,\\
|a(u,v)|&\leq M \|u\|_V\|v\|_V,  &{\mbox for~all~} u,v \in V,
\end{align*}
for some constants $\alpha$, $M>0$. Furthermore, there exists an increasing sequence $0= \lambda_0 < \cdots \leq \lambda_i \leq \cdots$ in $\IR$ and an orthonormal Hilbert basis of $H$ of eigenvectors $(\psi_i)_{i\in\IN}$ such that for all $i\in\IN$, $\psi_i \in V$ and $v\in V$ it is $a(\psi_i,v)=\lambda_i (\psi_i, v)$.
\end{lemma}

We next define the strong bidomain operator in the $L^q$-setting for $1<q<\infty$. We will use the same notation as for the weak setting since it will be clear from the context whether we consider the weak or strong formulation. To this end, let $L^q_{av} (\Omega):= \{u\in L^q(\Omega) : \int_\Omega u \; \d x = 0\}$ and let $P_{av}$ be the orthogonal projection from $L^q(\Omega)$ to $L^q_{av}(\Omega)$, i.e., $P_{av} u := u - \frac{1}{|\Omega|} \int_\Omega u \; \d x$. Then, we define the elliptic operators $A_i$ and $A_e$ by
\begin{align*}
A_{i,e} u  &:= - \divergence (\sigma_{i,e} \nabla u), \\
D(A_{i,e}) &:= \left\{ u \in W^{2,q} (\Omega) \cap L^q_{av}(\Omega) : \sigma_{i,e} \nabla u\cdot \nu =0 {\rm~a.e.~on~}\partial \Omega \right\} \subset L^q_{av}(\Omega).
\end{align*}
Here $A_{i,e}$ and $\sigma_{i,e}$ mean that either $A_i$ and $\sigma_i$ or $A_e$ and $\sigma_e$ are considered. Condition~\eqref{Eq: EV} implies that $D(A_i)=D(A_e)$. Hence, it is possible to define the sum $A_i+A_e$ with the domain $D(A_i)=D(A_e)$. Note that the inverse operator $(A_i+A_e)^{-1}$ on $L_{av}^q(\Omega)$ is a bounded linear operator.

Following~\cite{GK} we define the bidomain operator as follows. Let $\sigma_i$ and $\sigma_e$ satisfy Assumption~\ref{Ass: Conductivity matrices}. Then the bidomain operator $A$ is defined as 
\begin{align}
A = A_i (A_i+A_e)^{-1} A_e P_{av}
\end{align}
with domain 
\begin{align*}
D(A) := \{u \in W^{2,q}(\Omega) : \nabla u \cdot \nu = 0 {\rm~a.e.~on~} \partial \Omega\}.
\end{align*}

If we assume conservation of currents, i.e.,
\begin{align}
\label{Eq: Conversation of currents}
\int_{\Omega} (s_i(t)+s_e(t)) \; \d x =0, \quad t\geq 0
\end{align}
and moreover $\int_{\Omega} u_e \; \d x =0$, the bidomain equations~\eqref{Eq: BDE} may be equivalently rewritten as an evolution equation~\cite{MR2451724,GK} of the form 
\begin{align*}
\tag{ABDE}\label{Eq: ABDEIV}
\left\{
\begin{aligned}
\partial_t u + Au + f(u,w)&=s, \qquad &\mathrm{in} \ (0,\infty),\\
\partial_t w + g(u,w)&=0, &\mathrm{in} \ (0,\infty),\\
u (0)= u_0,~& ~w(0)= w_0, 
\end{aligned}
\right.
\end{align*}
where 
\begin{align}
\label{Eq: Modified source term}
 s:=s_i-A_i(A_i+A_e)^{-1}(s_i+s_e)
\end{align}
is the modified source term. The functions $u_e$ and $u_i$ can be recovered from $u$ by virtue of the following relations
\begin{align*}
u_e&=(A_i+A_e)^{-1}\{(s_i+s_e)-A_iP_{av}u\},\\
u_i&=u+u_e.
\end{align*}

\section{Weak time-periodic solutions}\label{Sec: Weak periodic solution}
In this section, we show the existence of weak time-periodic solutions by using a Galerkin approximation. 
We consider the bidomain equations under some growth conditions on $f$ and $g$ which contain the nonlinearities introduced in Section~\ref{Sec: Introduction}. 

We use the abstract form 

\begin{align*}
\tag{PABDE}\label{Eq: ABDE}
\left\{
\begin{aligned}
u' + Au + f(u,w)&=s, \qquad &\mathrm{in} \ \IT\times\Omega,\\
w' + g(u,w)&=0, &\mathrm{in} \ \IT\times\Omega,\\
u (t+T, x)= u(t,x),~& ~w(t+T)= w(t,x), 
\end{aligned}
\right.
\end{align*}
where $s$ is a $T$-periodic function for some $T>0$ and $\IT:= \IR/ T \IZ$ denotes the time torus. 
We assume that the nonlinear terms $f$ and $g$ satisfy the following conditions. 
\begin{assumption}{N}
\label{Ass: nonlinear}
Let $p>1$ be the number so that the Sobolev embedding $V\subset L^p(\Omega)$ holds. 
In other words, $2\le p$ if $d=2$; or $2\le p \le 6$ if $d=3$. 
The nonlinear terms $f,g:\IR\times\IR\to \IR$ are of the form 
\begin{align*}
f(u,w) &= f_1(u) + f_2(u)w, \\
g(u,w) &= g_1(u) + g_2 w, 
\end{align*}
where $g_2\in\IR$ and $f_1, f_2, g_1: \IR\to \IR$ are continuous functions. 
The functions are assumed to satisfy that there exist constants $C_0\in\IR$, $C_i>0$ $(i=1,\dots,5)$ and  $r>0$ such that
\begin{align}
C_0 + C_1|u|^p + C_2|w|^2 &\le r f(u,w)u + g(u,w)w \label{lower}\\
|f_1(u)| &\le C_3(1+ |u|^{p-1}) \label{growth1}\\
|f_2(u)| &\le C_4 (1+|u|^{p/2-1})\label{growth2}\\
|g_1(u)| &\le C_5(1+ |u|^{p/2})\label{growth3}
\end{align} 
for all $u,w \in \IR$. 
\end{assumption}
This assumption is a modified version of the assumption used in \cite{MR2451724}. 
Note that the three models introduced in Section~\ref{Sec: Introduction} hold the assumption as $p=4$. 
We shall check it in the appendix. 
We see that the following inequality holds: for any $(u,w)\in L^p(\Omega)\times H$, we have 
\begin{align*}
\|f(u,w)\|_{p'}^{p'} &\le C_6 (1+ \|u\|_p^p + \|w\|_H^2)\\
\|g(u,w)\|_H^2&\le C_7 (1+\|u\|_p^p +\|w\|_H^2)
\end{align*}
for some $C_i>0~(i=6,7)$ depending on $p$ and $C_3,\dots,C_5$, where $p'$ is the H\"older conjugate exponent, i.e., $1/p+1/p'=1$. 
In particular, $f(u,w) \in L^{p'}(Q)$ and $g(u,w) \in L^2(Q)$ for all $u\in L^p(Q), w\in L^2(Q)$. 
See in \cite[Lemma 25]{MR2451724}. 

Under this assumption, weak time-periodic solutions for (\ref{Eq: ABDE}) are defined as follows.

\begin{definition}\label{Def: Weak solution}
Let $T > 0$, $s \in L^2(\IT; V')$. 
Suppose that the Assumption \ref{Ass: nonlinear} holds. 
Then a pair of $(u, w)$ of $u: \IT \times \Omega \rightarrow \IR$, $w:\IT\times\Omega\rightarrow \IR$ is called a weak $T$-periodic solution to~\eqref{Eq: ABDE} if 
\begin{enumerate}[(i)]
\item $u \in C_w(\IT ; H) \cap L^2(\IT ; V)\cap L^p(Q)$, 
         $w \in C_w(\IT; H)$,\\
\item For all $\varphi_1 \in W^{1,2}(0,T; H) \cap L^2(0,T; V)\cap L^p(Q)$ and all $\varphi_2 \in W^{1,2}(0,T; H)$, 
\begin{align*}
\int_0^t \{(u,\partial_t \varphi_1) - a(u, \varphi_1) - {}_{p'}\langle f(u,w),\varphi_1\rangle_{p} \} \; \d \tau &= - \int_0^t {}_{V'}\langle s, \varphi_1\rangle_{V} \;\d \tau +(u(t),\varphi_1(t)) - (u(0),\varphi_1(0)),  \\
\int_0^t \{(w,\partial_t \varphi_2) - (g(u,w),\varphi_2) \} \; \d \tau &= (w(t),\varphi_2(t)) - (w(0),\varphi_2(0)),
\end{align*}
for all $t \in (0,T)$. 
Here $(\cdot, \cdot)$ denotes the $L^2$-inner product. 
\end{enumerate} 
A weak $T$-periodic solution $(u,w)$ is called strong if, in addition to above, it holds 
\begin{align*}
u\in W^{1,2}(\IT; H)\cap L^2(\IT; H^2(\Omega)), w\in W^{1,2}(\IT; H). 
\end{align*}
\end{definition}

The result on existence of weak time-periodic solutions reads as follows.

\begin{theorem}\label{Thm: Weak periodic solution}
Let $T>0$. For every $T$-periodic function $s \in L^2(\IT ; V')$ there exists at least one weak $T$-periodic solution $(u,w)$ to~\eqref{Eq: ABDE}.
\end{theorem}

\begin{proof}
Let $\{\psi_i\}_{i=0}^{\infty} \subset V$ be the orthonormal basis of eigenvectors of the bidomain bilinear form $a$ in $H$ 
and let $\{\lambda_i\}_{i=0}^{\infty} \subset \IR_{\ge0}$ be the corresponding eigenvalues, i.e., $a(U, \psi_i) = \lambda_i (U, \psi_i)_{L^2(\Omega)}$ for all $U\in V, i=0,1,\dots$. 
Let
\begin{align}
u_k(t,x) &:= \sum_{i=0}^k \alpha_{ki}(t)\psi_i(x),\label{Eq: Definition uk}\\
w_k(t,x)&:= \sum_{i=0}^k \beta_{ki}(t)\psi_i(x),\label{Eq: Definition wk}
\end{align}
with $\alpha_k(t) =\{\alpha_{kj}(t)\}_{j=0}^k$, $\beta_k(t)=\{\beta_{kj}(t)\}_{j=0}^k$, which are the solutions of the system of the ordinary differential equations
\begin{align}
\left\{
\begin{aligned}
 \frac{\d}{\d t}\alpha_{kj} &= -\alpha_{kj} \lambda_j - \int_{\Omega} f(u_k, w_k) \psi_j \; \d x + {}_{V'}\langle s(t), \psi_j\rangle_{V}, \\
\frac{\d}{\d t} \beta_{kj} &= -\int_{\Omega} g(u_k,w_k) \psi_j \; \d x, \\
\alpha_{kj}(0)&= a_j, \\
\beta_{kj}(0)&= b_j, 
\end{aligned}\label{Eq: Galerkinsystem}
\right.
\end{align}
for $j=0,1,\dots,k$. 
The initial data $\bla_k=\{a_j\}_{j=1}^k$ and $\blb_k=\{b_j\}_{j=1}^k$ are fixed later. 
By the standard theory of ordinary differential equations, this system admits a unique solution $(\alpha_k ,\beta_k) \subset (W^{1,2}(0,T_k))^{2(k+1)}$ on some interval $(0, T_k)$. 
It is either $|\alpha_k(t)| + |\beta_k(t)|\to \infty$ as $t\nearrow T_k$ or we can take any finite time $T_k$. 
In the following, it is shown that $|\alpha_k(t)| + |\beta_k(t)|\to \infty$ as $t\nearrow T_k$ does not occur by using a priori estimates. 
To this end, multiplying the first equation of~\eqref{Eq: Galerkinsystem} with $r\cdot\alpha_{kj}$, the second equation with $\beta_{kj}$, and summing over $j$ yield
\begin{align*}
&\frac{1}{2}\frac{\d}{\d t}\left(r\|u_k(t)\|_H^2 + \|w_k(t)\|_H^2\right) + r a(u_k(t) , u_k(t)) +  \int_{\Omega} r f(u_k(t),w_k(t))u_k(t)   + g(u_k(t),w_k(t))w_k(t) \, \d x \\&= r {}_{V'}\langle s(t), u_k(t)\rangle_{V}.
\end{align*}
We recall that the bidomain bilinear form $a$ has the coercivity of the form
\begin{align*}
\alpha \|U\|_V^2 \le a(U,U) + \alpha \|U\|_H^2
\end{align*}
for all $U\in V$ and for some constant $\alpha>0$, see~\cite{MR2451724}.  
By the coercivity of $a$, the Assumption \ref{Ass: nonlinear}, and Young's inequality, it is
\begin{align*}
&\frac{\d}{\d t}\left(r \|u_k(t)\|_H^2 +  \|w_k(t)\|_H^2\right) + C_{11} \|u_k(t)\|_V^2 + C_{12} \|u_k(t)\|_p^p - C_{13} \|u_k(t)\|_H^2  + C_{14}\|w_k(t)\|_H^2 \\
\leq & C_{15} \|s(t)\|_{V'}^2 + C_{16}, 
\end{align*}
for some constants $C_{1i}=C_{1i}(r,\alpha, C_j)>0~(i=1,\dots,6,\ j=0,\dots,2)$; we emphasize that all constants $C_{1i}$ are independent of $k$.  
We use the estimate 
\begin{align*}
C_{17}\|u_k(t)\|_p^p - C_{18} \le C_{12} \|u_k(t)\|_p^p - C_{13} \|u_k(t)\|_2^2
\end{align*}
for some $C_{17}, C_{18}>0$ since $2<p<\infty$. 
Therefore, we have the following estimate 
\begin{align}
&\frac{\d}{\d t}\left(r \|u_k(t)\|_H^2 +  \|w_k(t)\|_H^2\right) + C_{21} \left( r\|u_k(t)\|_V^2 + \|u_k(t)\|_p^p + \|w_k(t)\|_H^2\right) \nonumber \\
\leq & C_{22} \|s(t)\|_{V'}^2 + C_{23}, \label{ineq}
\end{align}
for some constants $C_{2i}>0~(i=1,2,3)$. 

Then, we apply Gronwall's inequality for the inequality (\ref{ineq}), then 
\begin{align}\label{Eq: Gronwall}
&r\|u_k(t)\|_H^2 + \|w_k(t)\|_H^2 \nonumber\\
\leq &\e^{-C_{21} t}( r\|\bla_k\|_H^2 +  \|\blb_k\|_H^2) +\int_0^t \e^{-C_{21}(t-\tau)}(C_{22}\|s(\tau)\|_{V'}^2 + C_{23}) \; \d \tau. 
\end{align}
Since $\|u_k(t)\|_H^2 = |\alpha_k(t)|^2$ and $\|w_k(t)\|_H^2 = |\beta_k(t)|^2$, this implies $T_k$ does not blow up at any finite time. 
We consider the Poincar\'e map 
\begin{align*}
\cS : \IR^{k+1}\times\IR^{k+1} \rightarrow \IR^{k+1}\times\IR^{k+1},\\
\cS (\bla_k,\blb_k) := (\alpha_k(T), \beta_k(T)).
\end{align*}
We define 
\begin{align*}
\IB_R:=\left\{(\bla_k,\blb_k)=(\{a_j\}_{j=0}^k,\{b_j\}_{j=0}^k)\in \IR^{k+1}\times \IR^{k+1}\middle| r \left(\sum_{j=0}^k|a_j|^2\right)^{1/2} + \left(\sum_{j=0}^k|b_j|^2\right)^{1/2} \le R\right\}
\end{align*}
with 
\begin{align}\label{R}
R^2 = \frac{\int_0^T \e^{-C_{21} (T-\tau)} (C_{22}\|s(\tau)\|_{V'}^2+C_{23})\, \d \tau}{1-\e^{-C_{21} T}}. 
\end{align}
Then, it follows that $\cS$ maps $\IB_R$ into itself from \eqref{Eq: Gronwall}. 
Since $\cS$ is also continuous, by Brouwer's fixed point theorem we conclude that $\cS$ admits a fixed point $(\bar{\bla}_k,\bar{\blb}_k)=\cS(\bar{\bla}_k,\bar{\blb}_k)$ in $\IB_R$ for all $k\in\IN$. 

In the following, we denote by $u_k$ and $w_k$ the functions defined in~\eqref{Eq: Definition uk} and~\eqref{Eq: Definition wk} respectively, corresponding to the solutions $\alpha_k$, $\beta_k$ of~\eqref{Eq: Galerkinsystem} with initial values $\bar{\bla}_k$, $\bar{\blb}_k$. 
Then, $u_k(0,x)=u_k(T,x)$ and $w_k(0,x)=w_k(T,x)$. 
Moreover, we see $u_k(t+T,x)= u_k(t,x)$ and $w_k(t+T,x)=w_k(t,x)$ for all $t \in\IR$ by periodically expansion. 
In the next step, we would like to pass to the limit $k\rightarrow \infty$ and show the existence of a weak solution to the original problem~(ABDE). 
To do so, we consider the uniform boundedness. 
For the inequality \eqref{Eq: Gronwall}, we take the supremum from $t=0$ to $t=T$, then by using (\ref{R})
\begin{align*}
\|u_k\|_{L^\infty(0,T; H)} + \|w_k\|_{L^\infty (0,T; H)} \le C_{31} \|s\|_{L^2(0,T; V')} + C_{32},  
\end{align*}
for some $C_{3i}>0~(i=1,2)$. 
Moreover, for the inequality (\ref{ineq}), integrate from $t=0$ to $t=T$ to get 
\begin{align}\label{energy ineq}
\|u_k\|^2_{L^2(0,T; V)} + \|u_k\|^p_{L^p(Q)} + \|w_k\|^2_{L^2(0,T; H)}\le C_{41} \|s\|^2_{L^2(0,T; V')} + C_{42}. 
\end{align}
for some $C_{4i}>0~(i=1,2)$. 
This implies that there are sub-sequences of $\{u_k\}_{k=1}^\infty$ and $\{w_k\}_{k=1}^\infty$, for convenience still denoted by $\{u_k\}_{k=1}^\infty$ and $\{w_k\}_{k=1}^\infty$, that converges to $u$ weakly in $L^2(0,T;V) \cap L^p(Q)$ and converges to $w$ weakly in $L^2(0,T;H)$. 

By construction of the function $u_k$ and $w_k$, 
\begin{align*}
(\partial_t u_k(t), \psi_\ell) + a(u_k(t),\psi_\ell) + {}_{p'}\langle f(u_k,w_k),\psi_\ell\rangle_p &= {}_{V'}\langle s(t), \psi_\ell\rangle_V\\
(\partial_t w_k(t), \psi_\ell) + (g(u_k,w_k),\psi_\ell) &= 0
\end{align*}
for all $\ell=0,\dots, k$. 
Integrate from $t_0$ to $t_1$ $(0\le t_0 \le t_1\le T)$, then we get 
\begin{align*}
&|(u_k(t_1), \psi_\ell) - (u_k(t_0),\psi_\ell)| \\
=&  |\int_{t_0}^{t_1} - a(u_k,\psi_\ell) - {}_{p'}\langle f(u_k,w_k),\psi_\ell\rangle_p + {}_{V'}\langle s, \psi_\ell\rangle_{V}\, \d \tau|\\
\le & \|u_k\|_{L^2(0,T; V)} \|\psi_\ell\|_{L^2(0,T; V)} + \|f(u_k,w_k)\|_{L^{p'}(Q)}\|\psi_\ell\|_{L^p(Q)} + \|s\|_{L^2(t_0,t_1;V')}\|\psi_\ell\|_{L^2(0,T; V)}\\
\le & C(|t_1-t_0|^{1/2} + |t_1-t_0|^{1/p} + \|s\|_{L^2(t_0,t_1;V')})\\
&|(w_k(t_1), \psi_\ell) - (w_k(t_0),\psi_\ell)| \\
=&  |\int_{t_0}^{t_1} - (g(u_k,w_k),\psi_\ell)\, \d \tau|\\
=& \|g(u_k,w_k)\|_{L^2(Q)} \|\psi\|_{L^2(Q)}\\
\le &C |t_1-t_0|^{1/2}
\end{align*}
for some $C=C(s,\psi_\ell)$ independent of $t_0, t_1$ and $k$, where we use the inequality (\ref{energy ineq}) and the embedding assumption $(\psi_\ell\in) V\subset L^p(\Omega)$. 
Therefore, it follows that for any $\varepsilon>0$ there is a $\delta>0$ with 
\begin{align*}
|(u_k(t_1), \psi_\ell) - (u_k(t_0),\psi_\ell)| + |(w_k(t_1), \psi_\ell) - (w_k(t_0),\psi_\ell)| < \varepsilon\quad {\rm if}~|t_1-t_0|\le \delta, ~k=1,2,\dots. 
\end{align*}
This means the families $\{(u_k(t), \psi_\ell)\}_{k=1}^\infty$ and $\{(w_k(t), \psi_\ell)\}_{k=1}^\infty$ are equicontinuous. 
Since $\{(u_k,\psi_\ell)\}_{k=1}^\infty$ and $\{(w_k,\psi_\ell)\}_{k=1}^\infty$ are uniform bounded in $k$, from Ascoli-Arzela's theorem, it follows that the subsequences $\{(u_k(t),\psi_\ell)\}_{k=1}^\infty$ and $\{(w_k(t),\psi_\ell)\}_{k=1}^\infty$ converge uniformly to continuous functions $(u(t), \psi_\ell)$ and $(w(t), \psi_\ell)$ for each fixed $\ell$. 
By the Cantor diagonalization argument and a density argument, this convergence can be generalized that for each $\psi\in H$, $\{(u_k(t),\psi)\}_{k=1}^\infty$ and $\{(w_k(t),\psi)\}_{k=1}^\infty$ converge uniformly to continuous functions $(u(t),\psi)$ and $(w(t),\psi)$. 
Therefore, we have $u\in C_w(\IT ; H)$ and $w\in C_w(\IT ; H)$. 

It remains to show the weak convergence of the nonlinear terms $f(u_k,w_k), g(u_k,w_k)$. 
We first prove $u_k\to u$ in $L^2(Q)$. 
To do so, we use Friedrich's inequality, which states that for any $\varepsilon>0$, there exists $J\in\IN$ and $\phi_1,\phi_2,\dots,\phi_J\in H$ such that for all $U\in V$, the following inequality holds
\begin{align*}
\|U\|_H^2 \le \sum_{j=1}^J \left|\int_\Omega U \phi_j \,\d x\right|^2 + \varepsilon \|\nabla U\|_H^2. 
\end{align*}
For Friedrich's inequality, see e.g. \cite{Galdi}. 
This inequality with $U=u_k - u$, the uniform boundedness of $\{u_k\}_{k=1}^\infty \subset L^2(0,T; V)$, and $u_k\to u$ in $C_w(\IT; H)$ implies that $u_k\to u$ in $L^2(Q)$. 
Since we have $u_k\to u$ a.e. in $Q$ and $f_1, f_2, g_1$ are continuous, $f_1(u_k)\to f_1(u), f_2(u_k) \to f_2(u), g_1(u_k)\to g(u)$ a.e. in $Q$ are satisfied. 
We shall show uniform boundedness in $L^{p'}(Q)$ for $f(u_k,w_k)$ and uniform boundedness in $L^2(Q)$ for $g(u_k,w_k)$, which implies $f(u_k,w_k) \to f(u,w)$ weakly in $L^{p'}(Q)$ and $g(u_k,w_k)\to g(u,w)$ weakly in $L^2(Q)$. 
Fortunately, under the Assumption \ref{Ass: nonlinear}, it has already been proved in \cite[p.477]{MR2451724}. 
Since the functions $u_k$, $w_k$ satisfy that for all $\varphi_1 \in H^1(0,T; H) \cap L^2(0,T; V)\cap L^p(Q)$ and all $\varphi_2 \in W^{1,2}(0,T; H)$, 
\begin{align*}
\int_0^t \{(u_k,\partial_t \varphi_1) - a(u_k, \varphi_1) - {}_{p'}\langle f(u_k,w_k),\varphi_1\rangle_{p} \} \; \d \tau &= - \int_0^t {}_{V'}\langle s, \varphi_1\rangle_{V} \;\d \tau +(u_k(t),\varphi_1(t)) - (u_k(0),\varphi_1(0)),  \\
\int_0^t \{(w_k,\partial_t \varphi_2) - (g(u_k,w_k),\varphi_2) \} \; \d \tau &= (w_k(t),\varphi_2(t)) - (w_k(0),\varphi_2(0)),
\end{align*}
for all $t \in (0,T)$, combining above discussions about the weak convergence, we show the existence of a weak $T$-periodic solution. 
\end{proof}

\section{Regularity of weak periodic solution}\label{Sec: Strong periodic solution}

In this section, we shall show that for the FitzHugh--Nagumo nonlinearities introduced in Section~\ref{Sec: Introduction} the weak time-periodic solution constructed in the previous section is actually a strong solution. 
In order to do so, we first review the global strong well-posedness result by Colli Franzone and Savar\'e~\cite{MR1944157}. 
After that, we use a weak-strong uniqueness argument to show the existence of a strong time-periodic solution for~\eqref{Eq: ABDE} with FitzHugh--Nagumo nonlinearities. 
In ~\cite{MR1944157}, they considered the initial boundary value problem for the bidomain equations of the form 
\begin{align*}
\tag{BDE II}\label{Eq: BDEII}
\left\{
\begin{aligned}
\partial_{t}u - \divergence (\sigma_i \nabla u_i) + f(u)+\theta w&=s_i, \qquad &\mathrm{in} \ (0,\infty)\times\Omega,\\
\partial_{t}u + \divergence (\sigma_e \nabla u_e) + f(u)+\theta w&=-s_e, \qquad &\mathrm{in} \ (0,\infty)\times\Omega,\\
u &= u_i - u_e, \qquad &\mathrm{in} \ (0,\infty)\times\Omega,\\
\partial_{t}w + \gamma w - \eta u&=0, &\mathrm{in} \ (0,\infty)\times\Omega,\\
\sigma_i \nabla u_i \cdot \nu = g_i, ~\sigma_e \nabla u_e\cdot \nu &= g_e, &\mathrm{in} \ (0,\infty)\times\partial\Omega,\\
u (0)= u_0,~& ~w(0)= w_0, &\mathrm{in}\ \Omega,  \\
\end{aligned}
\right.
\end{align*}
with $\theta, \gamma, \eta>0$. 

They regarded the bidomain equation as the {\it degenerate} variational formulation of the form 
\begin{align*}
\left\{
\begin{aligned}
(Bu)' + Au + \cF u& =L \hspace{5mm}  t \in (0,T)\\
(Bu)(0) &= \ell^0,  
\end{aligned}
\right.
\end{align*}
and constructed the global weak formulation and their regularity. 

Let $\Omega$ be a Lipschitz domain of $\IR^d$, $\Gamma:=\partial\Omega$, and the measurable function $\sigma_{i,e}: \overline{\Omega}\to\IR^{d\times d}$ satisfy the uniform ellipticity condition. 
Assume the nonlinear term $f$ is a continuous function with 
\begin{align}
f(0) = 0,~\exists\lambda_f \ge0: \frac{f(x)-f(y)}{x-y}\ge -\lambda_f,~~\forall x,y\in\IR,~{\rm with}~ x\neq y. \label{assumption(f)}
\end{align}
Their result is as follows. 

\begin{theorem}[Franzone-Savar\'e '02 \cite{MR1944157}]
Let us assume $s_{i,e}\in L^2(0,\cT; H)$, $g_{i,e} \in W^{1,1}(0,\cT; H^{-1/2}(\Gamma))$ satisfy $s_i + s_e \in W^{1,1}(0, \cT; H)$ and the compatibility condition 
\begin{align*}
\int_\Omega (s_i + s_e) \; \d x + {}_{H^{-1/2}(\Gamma)}\langle g_i + g_e, 1\rangle_{H^{1/2}(\Gamma)} = 0. 
\end{align*}
Then for any initial data $u_0 , w_0 \in H$, there uniquely exist a couple 
\begin{align*}
u_{i,e} \in L^2(0,\cT; V),~\int_\Omega u_e = 0~{\rm a.e~}t
\end{align*}
and 
\begin{align*}
&u \in C([0,\cT]; H) \cap L^2(0,\cT; V),~\partial_t u \in L^2_{loc}(0,\cT; H), \\
&f(u(t)) \in L^1(\Omega)\cap V'~a.e.~t\in(0,\cT), \\
&w, \partial_t w \in C([0,\cT]; H), 
\end{align*}
which solves the bidomain equation in the sense of 
\begin{align*}
&\int_\Omega (\partial_t u \hat{u} + \frac{\theta}{\eta} \partial_t w \hat{w}) \;\d x + \int_\Omega f(u) \hat{u} \;\d x + \sum_{i,e} \int_\Omega \sigma_{i,e}\nabla u_{i,e} \cdot \nabla \hat{u}_{i,e} \; \d x + \frac{\theta \gamma}{\eta} \int_\Omega w \hat{w} \; \d x + \theta \int_\Omega (w\hat{u} - u\hat{w}) \; \d x \\
&= \sum_{i,e} \int_\Omega s_{i,e} \hat{u}_{i,e} \;\d x + \sum_{i,e} {}_{H^{-1/2}(\Gamma)}\langle g_{i,e}, \hat{u}_{i,e}\rangle_{H^{1/2}(\Gamma)}, \\
&\int_\Omega (u(0) \hat{u} + \frac{\theta}{\eta} w(0)\hat{w}) \;\d x = \int_\Omega (u_0 \hat{u} + \frac{\theta}{\eta} w_0\hat{w}) \;\d x, 
\end{align*}
for a.e. $t\in (0,\cT)$ and all $\hat{u}_{i,e} \in V\times V$ with $\int_\Omega \hat{u}_e \; \d x=0$ and $\hat{u}=\hat{u}_i -\hat{u}_e$ and $\hat{w} \in H$. 

Moreover if $u_0 \in V, u_0f(u_0) \in L^1(\Omega)$, then 
\begin{align*}
u_{i,e} \in C([0,\cT]; V) ,~\partial_t u \in L^2(0,\cT; H),~w \in C([0,\cT]; V). 
\end{align*}
\end{theorem}

Furthermore they derived the regularity results. 

\begin{proposition}\label{Prop: Strong solution}
In addition to the assumption in the theorem, suppose that $d=3$, and the nonlinear term $f$ has a cubic growth at infinity, i.e., 
\begin{align*}
0 < \liminf_{|r|\to \infty} \frac{f(r)}{r^3} \le \limsup_{|r|\to \infty} \frac{f(r)}{r^3} < +\infty. 
\end{align*}
Then the bidomain equation admits a unique strong solution $u_{i,e}, u, w$. 
Moreover, it satisfies 
\begin{align*}
-\divergence (\sigma_{i,e}\nabla u_{i,e}) \in L^2(0,\cT; H)
\end{align*}
\end{proposition}

\begin{remark}
Let $\Omega$ be of class $C^{1,1}$, $\sigma_{i,e}$ be Lipschitz in $\Omega$ and $g_{i,e} \in L^2(0,\cT; H^{1/2}(\Gamma))$. 
Then by the standard regularity theorem, we see 
\begin{align*}
u_{i,e}\in L^2(0,\cT; H^2(\Omega)). 
\end{align*}
\end{remark}

\begin{remark}
If we look at the function $f$ of the FitzHugh--Nagumo nonlinearity introduced in Section~\ref{Sec: Introduction} as $f(u,w)= f(u)+w= u(u-a)(u-1)+w$, then the function $f(u)$ satisfies the assumptions for the nonlinearity in Proposition~\ref{Prop: Strong solution} as well as Assumption~\eqref{assumption(f)}.
\end{remark}


Now, we combine the results from the previous sections to obtain a strong time-periodic solution for the bidomain equations with FitzHugh--Nagumo type nonlinearities subject to arbitrary large forces. 
We would like to identify our {\it weak} time-periodic solution $(v,z)$ constructed in Section~\ref{Sec: Weak periodic solution} with a {\it strong} solution $(u,w)$ to the initial value problem with initial data $v(t_0), z(t_0)$ for some $t_0 > 0$ satisfying $v(t_0) \in V$ and $f(v(t_0))v(t_0) \in L^1(\Omega)$. 
Since $f$ has cubic growth, i.e., $p=4$ in the Assumption (\ref{Ass: nonlinear}), $v\in L^4(Q)$ derives the existence of $t_0$. 
So we can use the theorem by Colli-Franzone and Savar\'e for the global strong solution with the initial values $v(t_0)$, $z(t_0)$. 
Finally, we show that the weak solution $(v,z)$ coincides with the strong solution $(u,w)$ and therefore obtain the existence of a strong time-periodic solution. 
We follow the approach given in~\cite{GHK}

To be more precise, for given $T$-time-periodic functions $s_{i,e}\in L^2(0,\cT; H)$ with $s_i + s_e \in W^{1,1}(0, \cT; H)$ and $\int_\Omega (s_i + s_e) \; \d x= 0$ for a.e. $t$, let $(v,z)$ be a weak $T$-time-periodic solution of \eqref{Eq: ABDE} for $s=s_i - A_i(A_i+A_e)^{-1}(s_i+s_e)$ $(\in L^2(\IT; H))$  corresponding to Theorem \ref{Thm: Weak periodic solution}. 
We take $t_0$ such that $v(t_0) \in V$ and $v(t_0)f(v(t_0))\in L^1(\Omega)$. 
Since $(v,z)$ is a weak $T$-time-periodic solution, it satisfies that for all $\varphi_1 \in W^{1,2}(t_0,\cT; H) \cap L^2(t_0,\cT; V)\cap L^4(Q)$ and all $\varphi_2 \in W^{1,2}(t_0,\cT; H)$
\begin{align}
\int_{t_0}^t \{(v,\partial_t \varphi_1) - a(v, \varphi_1) - ( f(v,z),\varphi_1)  \} \; \d \tau &= - \int_{t_0}^t (s(\tau), \varphi_1(\tau)) \;\d \tau +(v(t),\varphi_1(t)) - (v(t_0),\varphi_1(t_0)),  \label{Eq: weak1}\\
\int_{t_0}^t \{(z,\partial_t \varphi_2) - (g(w,z) , \varphi_2) \} \; \d \tau &= (z(t),\varphi_2(t)) - (z(t_0),\varphi_2(t_0)),\label{Eq: weak2}
\end{align}
for all $t \in (t_0,\cT)$, and $(v,z)$ satisfies the following strong energy inequality: 
\begin{align}\label{Eq: strong energy inequality}
&\left(\|v(t)\|_H^2 + \|z(t)\|_H^2\right) + 2 \int_{t_0}^t a(v(\tau) , v(\tau)) \; \d \tau+ 2 \int_{t_0}^t \int_{\Omega} f(v(\tau),z(\tau))v(\tau) + g(v(\tau),z(\tau)) z(\tau) \; \d x \;\d \tau\nonumber \\
\le & \|v(t_0)\|_H^2 + \|z(t_0)\|_H^2 + 2\int_{t_0}^t (s(\tau), v(\tau)) \; \d \tau, 
\end{align}
for all $t\in[t_0, \cT]$. 

We next consider the unique global strong solution $(u,w)\in \left(W^{1,2}(t_0,\cT; H) \cap L^2(0,\cT; H^2(\Omega))\right) \times C^1([0,\cT]; H)$ corresponding to the initial-boundary value problem for the bidomain equation with initial value $(v(t_0), z(t_0))$ and $T$-periodic right-hand side $s_{i,e}$ and $g_{i,e}=0$. 
In the following, we show that the weak solution $(v,z)$ agrees with the strong solution $(u,w)$. 

Since $(u,w)$ is a strong solution, it satisfies that for all $\cT>t_0$ and all $\phi_1 \in W^{1,2}(t_0,\cT; H) \cap L^2(t_0,\cT; V)\cap L^4(Q)$ and all $\phi_2 \in W^{1,2}(t_0,\cT; H)$
\begin{align}
\int_{t_0}^t \{(u,\partial_t \phi_1) - a(u, \phi_1) - (f(u,w),\phi_1) \} \; \d \tau &= - \int_{t_0}^t ( s(\tau), \phi_1(\tau)) \;\d \tau +(u(t),\phi_1(t)) - (v(t_0),\phi_1(0)), \label{Eq: strong1} \\
\int_{t_0}^t \{(w,\partial_t \phi_2) - (g(u,w), \phi_2)\} \; \d \tau &= (w(t),\phi_2(t)) - (z(t_0),\phi_2(t_0)),\label{Eq: strong2}
\end{align}
for all $t \in (t_0,\cT)$, and $(u,w)$ satisfies the following strong energy identity: 
\begin{align}\label{Eq: strong energy equality}
&\left(\|u(t)\|_H^2 + \|w(t)\|_H^2\right) + 2 \int_{t_0}^t a(u(\tau) , u(\tau)) \; \d \tau+ 2 \int_{t_0}^t \int_{\Omega} f(u(\tau),w(\tau))u(\tau) + g(u(\tau),w(\tau))w(\tau) \; \d x \;\d \tau  \nonumber\\
= & \|v(t_0)\|_H^2 +\|z(t_0)\|_H^2+ 2\int_{t_0}^t (s(\tau), u(\tau))  \; \d \tau. 
\end{align}

Next, denote by
\begin{align*}
&v_h(t) := \int _0^{\cT} j_h(t-\tilde{t})v(\tilde{t}) \; \d \tilde{t}, \quad z_h(t) := \int_0^{\cT} j_h(t-\tilde{t}) z(\tilde{t}) \; \d \tilde{t}, \\
&u_h(t) := \int_0^{\cT} j_h(t-\tilde{t}) u(\tilde{t}) \; \d \tilde{t}, \quad w_h(t) := \int_0^{\cT} j_h(t-\tilde{t}) w(\tilde{t}) \; \d \tilde{t}
\end{align*}
the (Friedrichs) time-mollifier of $v$, $z$, $u$, and $w$, respectively, where $j_h \in \C_c^{\infty} (-h, h)$, $0<h<\cT$, is even and positive with $\int_{\IR} j_h(\tilde{t}) \; \d \tilde{t} =1$. 
Then, as is well known,
\begin{align}
\lim_{h \rightarrow 0} \int_0^{\cT} \| v_h(\tau) -v(\tau) \|_V^2 \; \d \tau = 0, \quad & \Esup_{t\in [0,\cT]} \| v_h(t)\|_2 \leq \Esup_{t\in [0, \cT]} \|v(t)\|_2, \\
\lim_{h \rightarrow 0} \int_0^{\cT} \| u_h(\tau) -u(\tau) \|_{H^2}^2 \; \d \tau = 0, \quad & \Esup_{t\in [0,\cT]} \| u_h(t)\|_V \leq \Esup_{t\in [0, \cT]} \|u(t)\|_V,\\
\lim_{h \rightarrow 0} \int_0^{\cT} \| z_h(\tau) -z(\tau) \|_H^2 \; \d \tau = 0, \quad & \lim_{h \rightarrow 0} \int_0^{\cT} \| w_h(\tau) -w(\tau) \|_{H}^2 \; \d \tau = 0. \quad & 
\end{align}
The weak continuity of $v$ and $u$ implies
\begin{align}
&\lim_{h \rightarrow 0} (u(t), v_h(t)) = \lim_{h \rightarrow 0} (u_h(t),v(t)) = (u(t), v(t)), \ t\geq t_0, \label{Eq: weak continuity u v}\\
&\lim_{h \rightarrow 0} (w(t), z_h(t)) = \lim_{h \rightarrow 0} (w_h(t),z(t)) = (w(t), z(t)), \ t\geq t_0. \label{Eq: weak continuity w z}
\end{align}
Furthermore since  
\begin{align*}
&\int_{t_0}^t (v,\partial_t u_h)\; \d \tau = - \int_{t_0}^t (u_h, \partial_t v)\; \d \tau + (u_h(t),v(t)) - (u_h(t_0), v(t_0)), \\
&\int_{t_0}^t (z, \partial_t w_h)\; \d \tau = - \int_{t_0}^t (w_h, \partial_t z)\; \d \tau + (w_h(t),z(t)) - (w_h(t_0), z(t_0)),  
\end{align*}
by taking the limit, 
\begin{align*}
&\lim_{h\to 0} \left\{\int_{t_0}^t (v,\partial_t u_h) +  (u_h, \partial_t v)\; \d \tau\right\} =  (u(t),v(t)) - \|v(t_0)\|_H^2 \\
&\lim_{h\to 0} \left\{\int_{t_0}^t (z, \partial_t w_h) + (w_h, \partial_t z)\; \d \tau \right\} =  (w(t),z(t)) - \|z(t_0)\|_H^2. 
\end{align*}

We now replace $\varphi_1$ by $u_h$ in~\eqref{Eq: weak1}, $\varphi_2$ by $w_h$ in~\eqref{Eq: weak2}, $\phi_1$ by $v_h$ in~\eqref{Eq: strong1}, and $\phi_2$ by $z_h$ in~\eqref{Eq: strong2}. 
Then, we sum up the resulting equations to obtain
\begin{align}
&\int_{t_0}^t \{-2 a(u,v) -  (f(v,z),u) -  (f(u,w),v) - (g(v,z),w) - (g(u,w),z)\}\; \d \tau \nonumber\\
= &- \int_{t_0}^t (s(\tau), u(\tau)+v(\tau)) \; \d \tau + (u(t),v(t))-\|v(t_0)\|_H^2+ (w(t),z(t))-\|z(t_0)\|_H^2. \label{crossterm}
\end{align}

To prove $(u,w)=(v,z)$, we calculate 
\begin{align*}
&\|u(t)-v(t)\|_H^2 + \|w(t)-z(t)\|_H^2 + 2 \int_{t_0}^t a(u(\tau)-v(\tau), u(\tau)-v(\tau)) \; \d \tau\\
=&\left(\|u(t)\|_H^2 + \|w(t)\|_H^2 + 2 \int_{t_0}^t a(u(\tau),u(\tau)) \; \d \tau\right) + \left(\|v(t)\|_H^2 + \|z(t)\|_H^2 + 2 \int_{t_0}^t a(v(\tau),v(\tau)) \; \d \tau\right)\\
&\hspace{5mm} - 2 (u(t), v(t)) - 2 (w(t),z(t)) - 4 \int_{t_0}^t a(u(\tau), v(\tau)) \; \d \tau. 
\end{align*}
For the first two parts, we use the strong energy equality~\eqref{Eq: strong energy equality} and  the strong energy inequality~\eqref{Eq: strong energy inequality}, and for the last term, we use the relation~(\ref{crossterm}). 
Then, we have 
\begin{align*}
&\|u(t)-v(t)\|_H^2 +  \|w(t)-z(t)\|_H^2 + 2 \int_{t_0}^t a(u(\tau)-v(\tau), u(\tau)-v(\tau)) \; \d \tau\\
&\le  2 \int_{t_0}^t \{(f(v,z),u)+(f(u,w),v)-(f(u,w),u)-(f(v,z),v)   \\
&\qquad \ \ \ + (g(v,z),w) + (g(u,w),z) - (g(u,w),w) - (g(v,z),z) \}\;\d\tau \\
\le & - 2 \int_{t_0}^t (f(u,w)-f(v,z), u-v) + (g(u,w)-g(v,z), w-z) \;\d\tau 
\end{align*}
Here, for the first term we use the Assumption~(\ref{assumption(f)}) and Young's inequality to get 
\begin{align*}
- 2 \int_{t_0}^t (f(u,w)-f(v,z), u-v)\;\d\tau &\le 2 \lambda_f \int_{t_0}^t \|u(\tau)-v(\tau)\|_H^2 \; \d \tau - 2\int_{t_0}^t (w-z, u-v) \: \d \tau \\
&\le 2 \lambda_f \int_{t_0}^t \|u(\tau)-v(\tau)\|_H^2 \; \d \tau + \int_{t_0}^t \varepsilon_1 \|w-z\|_H^2 +C(\varepsilon_1) \|u-v\|_H^2 \: \d \tau
\end{align*}
for some constants $\varepsilon_1$, $C(\varepsilon_1) >0$.
On the other hand since the function $g(u,w)=-\varepsilon(ku-w)$ is linear, 
\begin{align*}
\left|(g(u,w)-g(v,z), w-z)\right| \le C (\|u-v\|_H^2 + \|w-z\|_H^2). 
\end{align*}
for some $C>0$. 
Therefore, we have 
 \begin{align*}
 &\|u(t)-v(t)\|_H^2 +  \|w(t)-z(t)\|_H^2 + 2 \int_{t_0}^t a(u(\tau)-v(\tau), u(\tau)-v(\tau)) \; \d \tau\\
\le & C\int_{t_0}^t (\|u(\tau)-v(\tau)\|_H^2 + \|w(\tau)-z(\tau)\|_H^2 \; \d\tau, 
\end{align*}
for some $C>0$, which is different from the previous constant. 

Hence, we are able to apply Gronwall's lemma to conclude that 
\begin{align*}
u-v \equiv0, w-z\equiv0~{\rm a.e.~in~}\Omega\times [t_0, \cT]. 
\end{align*}
This implies the existence of a strong $T$-time-periodic solution $(u,w)$ when the source term $s_{i,e}$ is a $T$-time-periodic function. 

We write down the main theorem of the existence of strong periodic solutions without assuming smallness conditions for the external forces. 

\begin{theorem}\label{Thm: Maintheorem}
Let $d=3$, $T>0$, and $s_{i,e} \in L^2(\IT;H)$ with $s_i + s_e \in W^{1,1}(\IT; H)$ and $\int_\Omega (s_i + s_e) \; \d x=0$ for a.e. $t$. 
Let the conductivity matrices $\sigma_{i,e}$ satisfy the Assumption \ref{Ass: Conductivity matrices} and the nonlinear term $f$ satisfy the  Assumption~(\ref{assumption(f)}) and assume that there exist constants $C_0\in \IR$ and $C_1>1$ such that 
\begin{align}
C_0 + C_1|u|^4 \le f(u)u \label{lower2}
\end{align}
for all $u\in \IR$. 
Then for the bidomain equations with FitzHugh--Nagumo type 
\begin{align*}
\left\{
\begin{aligned}
\partial_{t}u - \divergence (\sigma_i \nabla u_i) + f(u)+ w&=s_i, \qquad &\mathrm{in} \ (0,\infty)\times\Omega,\\
\partial_{t}u + \divergence (\sigma_e \nabla u_e) + f(u)+ w&=-s_e, \qquad &\mathrm{in} \ (0,\infty)\times\Omega,\\
u &= u_i - u_e, \qquad &\mathrm{in} \ (0,\infty)\times\Omega,\\
\partial_{t}w -\varepsilon(ku -w) &=0, &\mathrm{in} \ (0,\infty)\times\Omega,\\
\sigma_i \nabla u_i \cdot \nu = 0, ~\sigma_e \nabla u_e\cdot \nu &= 0, &\mathrm{in} \ (0,\infty)\times\partial\Omega,\\
u (0)= u_0,~& ~w(0)= w_0, &\mathrm{in}\ \Omega,  \\
\end{aligned}
\right.
\end{align*}
there exists a strong $T$-periodic solution 
\begin{align*}
(u_i, u_e) \in (W^{1,2}(\IT;H) \cap L^2(\IT; H^2(\Omega))^2~{\mbox with~}\int_\Omega u_e \; \d x=0~{\rm a.e.}~t\\
(u,w) \in (W^{1,2}(\IT;H) \cap L^2(\IT; H^2(\Omega)) \cap L^4(\IT\times\Omega)) \times C^1(\IT; H). 
\end{align*}
\end{theorem}

\begin{remark}
The Assumption \ref{Ass: nonlinear} of the existence of the weak periodic solutions is replaced by (\ref{lower2}). 
\end{remark}

\begin{remark}
We do not treat the ionic models by Rogers--McCulloch and Aliev-Panfilov due to the lack of a suitable global well-posedness result for the initial value problem in the $L^2$ setting.
\end{remark}

\section{Appendix}
In this appendix, we check that the three models introduced in Section \ref{Sec: Introduction} satisfy the Assumption \ref{Ass: nonlinear}. 
Since the growth conditions (\ref{growth1})-(\ref{growth3}) are trivial as $p=4$, we confirm the condition (\ref{lower}). 
\subsection{FitzHugh--Nagumo model}
The FitzHugh--Nagumo type is 
\begin{align*}
f(u,w) &= u(u-a)(u-1) + w \\
g(u,w) &=-\varepsilon(ku - w) 
\end{align*}
with $0<a<1$ and $ k, \varepsilon>0$. Then, we are able to calculate as follows $(r=1)$: 
\begin{align*}
f(u,w)u + g(u,w)w = u^4 - (a+1)u^3 + au^2 + uw - \varepsilon k uw + \varepsilon w^2
\end{align*}
and by 
\begin{align*}
|(a+1) u^3| &\le \frac{1}{8}u^4 + c_{11},\\
|a u^2| &\le \frac{1}{8} u^4 + c_{12},\\
|uw| & \le \frac{1}{8} u^4 + \frac{\varepsilon}{4} w^2 + c_{13},\\
|\varepsilon uw| & \le \frac{1}{8}u^4 + \frac{\varepsilon}{4} w^2 + c_{14}, 
\end{align*}
for some $c_{1i}>0~(i=1,\dots,4)$, we have 
\begin{align*}
f(u,w)u + g(u,w)w \ge \frac{1}{2} u^4 + \frac{\varepsilon}{2} w^2 + c_1
\end{align*}
for some $c_1\in\IR$. 
Therefore, the FitzHugh--Nagumo model satisfies the Assumption \ref{Ass: nonlinear}. 

\subsection{Rogers--McCulloch model}
The Rogers--McCulloch type is 
\begin{align*}
f(u,w) &= bu(u-a)(u-1) + uw \\
g(u,w) &=-\varepsilon(ku - w) 
\end{align*}
with $0<a<1$ and $b,k,\varepsilon>0$. Then, we are able to calculate as follows: 
\begin{align*}
rf(u,w)u + g(u,w)w = rbu^4 - rb(a+1)u^3 + rbau^2 + ru^2w - \varepsilon k uw + \varepsilon w^2
\end{align*}
and, based on the calculation 
\begin{align*}
|ru^2w| \le \frac{C^2}{2}u^4 + \frac{r^2}{2C^2}w^2, 
\end{align*}
we choose $r,C>0$ depending on $b,\varepsilon$, such that
\begin{align*}
\begin{cases}
c_{21}:=rb - \frac{C^2}{2}>0, \\
c_{22}:=\varepsilon-\frac{r^2}{2C^2}>0. 
\end{cases}
\end{align*}
By 
\begin{align*}
|rb(a+1) u^3| &\le \frac{c_{21}}{6}u^4 + c_{23},\\
|rba u^2| &\le \frac{c_{21}}{6} u^4 + c_{24},\\
|\varepsilon kuw| & \le \frac{c_{21}}{6}u^4 + \frac{c_{22}}{2} w^2 + c_{25}, 
\end{align*}
for some $c_{2i}>0~(i=3,\dots,5)$, we have 
\begin{align*}
rf(u,w)u + g(u,w)w \ge \frac{c_{21}}{2} u^4 + \frac{c_{22}}{2} w^2 + c_2
\end{align*}
for some $c_2\in\IR$. 
Therefore, the Rogers--McCulloch model satisfies the Assumption \ref{Ass: nonlinear}. 

\subsection{Aliev--Panfilov model}
The modified Aliev--Panfilov type is 
\begin{align*}
f(u,w) &= bu(u-a)(u-1) + uw \\
g(u,w) &=\varepsilon(ku(u-1-d) + w) 
\end{align*}
with $0<a,d<1$, $b,k,\varepsilon>0$, and $b>k$. Then, we are able to calculate as follows: 
\begin{align*}
rf(u,w)u + g(u,w)w = rb u^4 + rb(a+1)u^3 + rbau^2 + ru^2 w + \varepsilon k u ^2 w - \varepsilon k (1+d) uw + \varepsilon w^2
\end{align*}
and, based on the calculation 
\begin{align*}
|(r+\varepsilon k) u^2w| \le \frac{C^2}{2}u^4 + \frac{(r+\varepsilon k)^2}{2C^2}w^2, 
\end{align*}
we choose $r, C >0$ depending on $b$, $k$, $\varepsilon$, such that
\begin{align*}
\begin{cases}
c_{31}:=rb - \frac{C^2 }{2}>0, \\
c_{32}:=\varepsilon-\frac{(r+\varepsilon k)^2}{2C^2}>0. 
\end{cases}
\end{align*}
Here, the assumption $b>k$ is essential.
By 
\begin{align*}
|rb(a+1) u^3| &\le \frac{c_{31}}{6}u^4 + c_{33},\\
|rba u^2| &\le \frac{c_{31}}{6} u^4 + c_{34},\\
|\varepsilon k(1+d)uw| & \le \frac{c_{31}}{6}u^4 + \frac{c_{32}}{2} w^2 + c_{35}, 
\end{align*}
for some $c_{3i}>0~(i=3,\dots,5)$, we have 
\begin{align*}
rf(u,w)u + g(u,w)w \ge \frac{c_{31}}{2} u^4 + \frac{c_{32}}{2} w^2 + c_3
\end{align*}
for some $c_3\in\IR$. 
Note that we are not able to take $c_{3i}~(i=1,2)$ in the case $b=k$. 
Therefore, the modified Aliev-Panfilov model satisfies the Assumption \ref{Ass: nonlinear}.


\begin{bibdiv}
\begin{biblist}

\bib{AP}{article}{
author={Aliev, R.}, 
author={Panfilov, A.},
title={A simple two-variable model of cardiac excitation}, 
date={1996}, 
journal={Chaos, Solitions \& Fractals.}, 
volume={7}, 
number={3}, 
pages={293\ndash301}, 
} 

\bib{BC}{unpublished}{
author={Belhachmi, Z.}, 
author={Chill, R.}, 
title={The bidomain problem as a gradient system}, 
date={2018}, 
note={arXiv: 1804.08272}, 
}

\bib{MR2451724}{article}{
      author={Bourgault, Y.},
      author={Coudi{\`e}re, Y.},
      author={Pierre, C.},
       title={{Existence and uniqueness of the solution for the bidomain model
  used in cardiac electrophysiology.}},
        date={2009},
     journal={Nonlinear Anal. Real World Appl.},
      volume={10},
      number={1},
       pages={458\ndash 482},
}

\bib{franzone1990mathematical}{article}{
      author={{Colli Franzone}, P.},
      author={Guerri, L.},
      author={Tentoni, S.},
       title={{Mathematical modeling of the excitation process in myocardial
  tissue: influence of fiber rotation on wavefront propagation and potential
  field.}},
        date={1990},
     journal={Math. Biosci.},
      volume={101},
      number={2},
       pages={155\ndash 235},
}

\bib{MR3308707}{book}{
      author={{Colli Franzone}, P.},
      author={Pavarino, L.~F.},
      author={Scacchi, S.},
       title={{Mathematical cardiac electrophysiology.}},
      series={{MS\&A. Modeling, Simulation and Applications}},
   publisher={Springer, Cham},
        date={2014},
      volume={13},
}

\bib{PSC05}{article}{
      author={{Colli Franzone}, P.},
      author={Pennacchio, M.},
      author={Savar{\'e}, G.},
       title={{Multiscale modeling for the bioelectric activity of the
  heart.}},
        date={2005},
     journal={SIAM J. Math. Anal.},
      volume={37},
      number={4},
       pages={1333\ndash 1370},
}

\bib{MR1944157}{incollection}{
      author={{Colli Franzone}, P.},
      author={Savar{\'e}, G.},
       title={{Degenerate evolution systems modeling the cardiac electric field
  at micro- and macroscopic level.}},
        date={2002},
   booktitle={{Evolution equations, semigroups and functional analysis
  ({M}ilano, 2000)}},
      series={{Progr. Nonlinear Differential Equations Appl.}},
      volume={50},
   publisher={Birkh{\"a}user, Basel},
       pages={49\ndash 78},
}

\bib{fitzhugh1961impulses}{article}{
      author={FitzHugh, R.},
       title={{Impulses and {P}hysiological {S}tates in {T}heoretical {M}odels
  of {N}erve {M}embrane.}},
        date={1961},
     journal={Biophys. J.},
      volume={1},
      number={6},
       pages={445\ndash 466},
}

\bib{HKKT}{article}{
      author={Hieber, M.},
      author={Kajiwara, N.},
      author={Kress, K.},
      author={Tolksdorf, P.},
       title={{Strong Time Periodic Solutions to the Bidomain Equations with FitzFugh-Nagumo Type Nonlinearities. }},
     journal={arXiv:1708.05304},
      date={2017},
}

\bib{Galdi}{book}{,
    AUTHOR = {Galdi, G. P.},
     TITLE = {An introduction to the mathematical theory of the
              {N}avier-{S}tokes equations},
    SERIES = {Springer Monographs in Mathematics},
   EDITION = {Second},
      NOTE = {Steady-state problems},
 PUBLISHER = {Springer, New York},
      YEAR = {2011},
     PAGES = {xiv+1018},
      ISBN = {978-0-387-09619-3},
   MRCLASS = {35Q30 (35-02 76D03 76D05 76D07)},
  MRNUMBER = {2808162},
       DOI = {10.1007/978-0-387-09620-9},
       URL = {https://doi.org/10.1007/978-0-387-09620-9},
}

\bib{GK}{article}{
      author={Giga, Y.},
      author={Kajiwara, N.},
       title={{On a resolvent estimate for bidomain operators and its
  applications.}},
          date={2018},
     journal={J. Math. Anal. Appl.},
     volume={459}, 
     number={1}, 
     pages={528\ndash 555}, 
}

\bib{GHK}{article}{
      author={Galdi, G.~P.},
      author={Hieber, M.},
      author={Kashiwabara, T.},
       title={{Strong time-periodic solutions to the 3D primitive equations subject to arbitrary large forces}},
          date={2017},
     journal={Nonlinearity},
     volume={30}, 
     number={10}, 
     pages={3979\ndash 3992}, 
}

\bib{GO}{article}{
      author={Gani, M.~.O.},
      author={Ogawa, T.},
       title={{Stability of periodic traveling waves in the Aliev--Panfilov reaction-diffusion system.}},
          date={2016},
     journal={Commun. Nonlinear Sci. Numer. Simul.},
     volume={33},  
     pages={30\ndash 42}, 
}

\bib{HP}{unpublished}{
author={Hieber, M.}, 
author={Pr\"uss, J.}, 
title={On the bidomain problem with FitzHugh--Nagumo transport. }, 
date={2018}, 
note={preprint}, 
}

\bib{HP2}{unpublished}{
author={Hieber, M.}, 
author={Pr\"uss, J.}, 
title={$L_q$-theory for the bidomain operator.}, 
date={2018}, 
note={preprint}, 
}

\bib{Kajiwara}{unpublished}{
      author={Kajiwara, N.},
      title={{Global strong solutions for the bidomain equations with maximal $L_{p,\mu}$ regularity. }},
        date={2018},
     note={preprint},
}

\bib{MR1673204}{book}{
      author={Keener, J.},
      author={Sneyd, J.},
       title={{Mathematical physiology.}},
      series={{Interdisciplinary Applied Mathematics}},
   publisher={Springer-Verlag, New York},
        date={1998},
      volume={8},
}

\bib{KW13}{article}{
      author={Kunisch, K.},
      author={Wagner, M.},
       title={{Optimal control of the bidomain system ({II}): uniqueness and
  regularity theorems for weak solutions.}},
        date={2013},
     journal={Ann. Mat. Pura Appl. (4)},
      volume={192},
      number={6},
       pages={951\ndash 986},
}

\bib{MM16}{article}{
      author={Mori, Y.},
      author={Matano, H.},
       title={{Stability of front solutions of the bidomain equation.}},
        date={2016},
     journal={Comm. Pure Appl. Math.},
      volume={69},
      number={12},
       pages={2364\ndash 2426},
}

\bib{rogers1994collocation}{article}{
      author={Rogers, J.~M.},
      author={McCulloch, A.~D.},
       	title = {A collocation-{G}alerkin finite element model of cardiac action potential propagation.},
        date={1994},
     journal={IEEE Trans. Biomed. Eng.},
      volume={41},
      number={8},
       pages={743\ndash 757},
}

\bib{Tun78}{article}{
      author={Tung, L.},
      title={{A bidomain model for describing ischemic myocardial d-c potentials.}},
        date={1978},
     journal={PhD Thesis, MIT.},
}

\bib{MR2474265}{article}{
      author={Veneroni, M.},
       title={{Reaction-diffusion systems for the macroscopic bidomain model of
  the cardiac electric field.}},
        date={2009},
     journal={Nonlinear Anal. Real World Appl.},
      volume={10},
      number={2},
       pages={849\ndash 868},
}

\end{biblist}
\end{bibdiv}
\end{document}